\documentclass[12pt, reqno]{amsart}
\usepackage{amsmath, amstext, amsbsy, amssymb, amscd}
\usepackage{amsxtra}
\usepackage{amscd}
\usepackage{amsthm}
\usepackage{amsfonts}
\usepackage{eucal}
\usepackage{color}
\usepackage[all]{xy}
\usepackage[CJKbookmarks=true]{hyperref}

\newtheorem{theorem}{Theorem}[section]
\newtheorem{lemma}[theorem]{Lemma}
\newtheorem{prop}[theorem]{Proposition}
\newtheorem{corollary}[theorem]{Corollary}
\theoremstyle{definition}

\newtheorem{remark}[theorem]{Remark}

\numberwithin{equation}{section}

\def\ggg{\mathfrak{g}}

\makeatletter
\newcommand{\rmnum}[1]{\romannumeral #1}
\newcommand{\Rmnum}[1]{\expandafter\@slowromancap\romannumeral #1@}
\makeatother

{\vskip-\lastskip\medskip
  \noindent
  {\em #1.}\enspace
  }%
{\qed\par\medskip
  }

\begin{document}

\title[ On  the Premet conjecture for finite W-superalgebras ]{On the Premet conjecture for  finite W-superalgebras }

\author{Husileng Xiao}

\address{ School of mathematical Science, Harbin Engineering University, Harbin, 15001, China }\email{hslxiao@hrbeu,edu.cn}
\begin{abstract}
Let $\bullet^{\dag}$ be the map constructed by Losev, which sends the set of two sided ideals of a finite W-algebras to that of the universal enveloping algebra of corresponding Lie algebras. The Premet conjecture which was proved in \cite{Lo11}, says that, restricted to the set of primitive ideals with finite codimension,  any fiber of the map $\bullet^{\dag}$ is a single orbit under an action of a finite group $C_e$.  In this article we formulate and prove a similar fact in the super case.  This will give  a classification to  the set of finite dimensional irreducible representations of W-superalgebras provided $C_{e}$ is a trivial group and the set of  primitive ideals of  the corresponding universal enveloping   algebra of  Lie superalgebra  is  known. 
\end{abstract}
\maketitle
\section{introduction}

The finite W-superalgebras   are mathematically a super generalization of the finite W-algebras.   They  are closely related to the supersymmetric theories in physics as Lie superalgebras.
Generalizing the groundbreaking work \cite{Pr1} of  Premet in non-super case,  Wang and Zhao in \cite{WZ}  gave a mathematical definition of  finite W-superalgebras from modular representation theory of Lie superalgebras. In the present paper we study finite dimensional irreducible  presentations of  W-superalgebras.
In  \cite{BBG} and \cite{BG},  the authors  establish  a Yangian presentation of  the W-superalgebra corresponding   to a principal nilpotent orbit in the general Lie superalgebras.
Relying on this explicit presentation, they gave  a  description of $\mathrm{Irr}^{\mathrm{fin}}(\mathcal{W})$ and further detailed information on the highest weight structure. 
Amitsur-Levitzki identity was proved   for W-superalgebras  associated with principal nilpotent orbits for $Q(N)$  in \cite{PS1}. Then the authors obtain that any irreducible
representation of the W-superalgebra  is finite dimensional. Those  results seem to  indicate that the representation theory of the  finite W-superalgebras is quite different from that of finite W-algebras.
 Presenting  a set  of generators  of  the W-superalgebras associated to the minimal nilpotent orbits,  Zeng and Shu  constructed irreducible representations of them with  dimension 1 or 2, see \cite{ZS}.
 However unlike in the case of finite W-algebras,  some fundamental problems in representation theory for general finite W-superalgebras  are  still open. 
 In \cite{SX}  the authors generalize the  Losev's  Poisson geometric approach to the super case and made a  step to give classification of finite dimensional irreducible representation of  for general finite W-superalgebras.  In this article we make a progress to this problem, by proving  Premet's  conjecture for Lie superalgebras of  basic types.   We hope that the readers  could feel  from here that the difference between  representation theory of finite W-algebras and   W-superalgebras probably   not exceeds the difference  between representation theory of Lie algebras and Lie superalgebras.

Let $\ggg=\ggg_{\bar{0}}\oplus \ggg_{\bar{1}}$ be a basic Lie superalgebra over an algebraically closed field $\mathbb{K}$, $\mathcal{U}$ and $\mathcal{U}_{0}$ be the enveloping algebra of $\mathfrak{g}$ and $\ggg_{\bar{0}}$ respectively.  Denote by $(\bullet, \bullet)$ the Killing form on it. Let $e \in \ggg_{\bar{0}}$ and $\chi \in \ggg_{\bar{0}}^{*}$ be the corresponding element to $e$ via the the Killing form. Pick an $\mathfrak{sl}_2$-triple $\{ f,h,e \} \subset \ggg_{\bar{0}}$ and let $\ggg=\oplus_{i}\ggg(i)$ (\textit{resp}. $\ggg_{\bar{0}}=\oplus_{i}\ggg(i)\cap\ggg_{\bar{0}}$) be the corresponding good $\mathbb{Z}$-grading. Denote by $\mathcal{W}$ and $\mathcal{W}_{0}$ the W-algebras associated to the pairs $(\ggg,e)$ and $(\ggg_{\bar{0}},e)$ respectively.  Let $\tilde{\mathcal{W}}$ be the extended W-superalgebra defined in $\mathcal{A}_{\ddag}$ \S3 \cite{SX}( or $\mathcal{A}_{\dag}$ \S6 \cite{Lo15}).
It was  obtained in \cite{SX} that there is a following relation among the three kind of W-algebras. (1), we have embedding
$\mathcal{W}_{0} \hookrightarrow \tilde{\mathcal{W}}$ and the later is generated over the former by $\dim(\ggg_{\bar{1}})$ odd elements.
(2), we have decomposition $\tilde{\mathcal{W}}=\mathrm{Cl}(V_{\bar{1}})\otimes \mathcal{W}$ of associative algebras, where $\mathrm{Cl}(V_{\bar{1}})$ is the Clifford algebra over a vector space $V_{\bar{1}}$ with a non-degenerate symmetric two form, see Theorem \ref{them1} for the details.
 Essentially, as we pointed out in \cite{SX},  this  makes $\mathcal{W}_{0}$  to  $\mathcal{W}$ as $\mathcal{U}_{0}$ to $\mathcal{U}$.

For a given associative algebra $A$, denote by $\mathfrak{id}(A)$  the set of two sided ideals of $A$ and by $\mathrm{Prim}^{\mathrm{cofin}}(A)$  the set of primitive ideals of $A$ with finite codimension in $A$. It is well known that $\mathrm{Prim}^{\mathrm{cofin}}(A)$ is bijective with the set $\mathrm{Irr}^{\mathrm{fin}}(A)$ of isomorphism classes of finite dimensional irreducible $A$ modules.
In \cite{Los10} Losev  constructed a ascending map $\bullet^{\dag} : \mathfrak{id}(\mathcal{W}_{0}) \longrightarrow \mathfrak{id}(\mathcal{U}_{0})$ and descending map $\bullet_{\dag} : \mathfrak{id}(\mathcal{U}_{0}) \longrightarrow \mathfrak{id}(\mathcal{W}_{0})$. Those two maps are crucial to his study on representations of $\mathcal{W}_{0}$.
The ascending map $\bullet^{\dag}$ sends $\mathrm{Prim}^{\mathrm{cofin}}(\mathcal{W}_{0})$ to the set $\mathrm{Prim}_{\mathbb{O}}(\mathcal{U}_{0})$ of primitive ideals of $\mathcal{U}_{0}$ supported on the Zariski closure of the adjoint orbit $\mathbb{O}=G_{\bar{0}} \cdot e$.
 Denote by  $Q=Z_{G_{\bar{0}}}\{ e,h,f \}$ the stabilizer of the triple $\{ e,h,f \}$ in $G_{\bar{0}}$ under the adjoint action.
 Let $C_{e}=Q/Q^\circ$, where $Q^\circ$ is the identity component of $Q$.  The Premet conjecture which was proved in \cite{Lo11}, is saying that for any
 $\mathcal{I} \in \mathrm{Prim}_{\mathbb{O}}(\mathcal{U}_{0}) $ the set $\{ \mathcal{I} \mid  \mathcal{I} \in \mathrm{Prim}^{\mathrm{fin}} (\mathcal{W}) ,  \mathcal{J}^{\dag}=\mathcal{I} \}$ is a single $C_{e}$ orbit.
 This gives an  almost complete classification of  $\mathrm{Irr}^{\mathrm{fin}}(\mathcal{W}_{0})$.

In this paper we generalize the above fact to the super case. The super analogue of the maps $\bullet^{\dag}$ and $\bullet_{\dag}$ were established in \cite{SX}. By abuse of notation, we also denote it by $\bullet^{\dag}$ and $\bullet_{\dag}$ from now on.
Denote by $\mathrm{Prim}_\mathbb{O}(\mathcal{U})$ the set of primitive ideals of $\mathcal{U}$ supported on the Zariski closure of the adjoint orbit $\mathbb{O}=G_{\bar{0}} \cdot e$, see \S 2 for the precise meaning  of  the term `supported' in the super context.
In \S 2 we will construct an action of $Q$  on $\tilde{\mathcal{W}}$ with a property that $Q^{\circ}$ leaves any two sided ideal of $\tilde{\mathcal{W}}$ stable, see Proposition \ref{prop2.1}. This provide us  an action of $C_e$ on $\mathfrak{id}(\tilde{\mathcal{W}})$.
The main result of the present paper is following.
\begin{theorem} \label{mainthm}
For any $\mathcal{J} \in \mathrm{Prim}_\mathbb{O}(\mathcal{U})$,  the set $\{ \mathrm{Cl}(V_{\bar{1}}) \otimes \mathcal{I} \mid  \mathcal{I} \in \mathrm{Prim}^{\mathrm{fin}} (\mathcal{W}) ,\quad  \mathcal{J}^{\dag}=\mathcal{I} \}$ is a single $C_e$-orbit.
\end{theorem}

Our strategy to prove the theorem is that we apply Theorem 4.1.1 \cite{Lo11} to the Harish-Chandra bimodule module $\mathcal{U}$ over $\mathcal{U}_{0}$
and the relation among $\mathcal{W}$, $\mathcal{W}_{0}$ and $\tilde{\mathcal{W}}$ obtained in Theorem 3.11 \cite{SX}. Our approach is
highly inspired by \S 6 \cite{Lo15}.

We can recover $\mathcal{I}$  from  $\mathrm{Cl}(V_{\bar{1}}) \otimes \mathcal{I}$ by Corollary \ref{coro2.4}.
It was proved in Theorem 4.8 \cite{SX} that the map $\bullet^\dag$ sends $\mathrm{Prim}^{\mathrm{cofin}}(\mathcal{W})$ to $\mathrm{Prim}_{\mathbb{O}}(\mathcal{U})$.  Thus Theorem \ref{mainthm}  almost completely reduced the problem of classifying  $\mathrm{Prim}^{\mathrm{fin}} (\mathcal{W})=\mathrm{Irr}^{\mathrm{fin}}(\mathcal{W})$ to that of $\mathrm{Prim}(\mathcal{U})$.   Provided  $\mathrm{Prim}(\mathcal{U})$  is known and $C_e$ is trivial, Theorem \ref{mainthm}  gives  a description of   $\mathrm{Irr}^{\mathrm{fin}}(\mathcal{W})$, see \S \ref{2.4}.   For the recent progress  on  primitive ideals of Lie superalgebras,  see \cite{CM} and \cite{Mu}, for example.

\section{Proof on the main result}

We first recall the Poisson geometric realization of finite W-(super)algebra in the sense of Losev.
Denote by $A_{0}$ (\textit{resp.} $A$) the Poisson (\textit{resp.} super) algebra
$S[\ggg_{\bar{0}}]$(\textit{resp. $S[\ggg]$}) with the standard bracket $\{ ,\}$ given by $\{x,y\}=[x,y]$ for all $x,y \in \ggg_{\bar{0}}$ (\textit{resp.} $\ggg$).
Let $\hat{A}_{0}$ (\textit{resp.} $\hat{A}$) be the completion of $A_{0}$ (\textit{resp.} $A$) with respect to the point $\chi \in \ggg_{\bar{0}}^{*}$(\textit{resp.} $\ggg$).
Let  $\mathcal{U}_{\hbar,0}^{\wedge}$ (\textit{resp.}  $\mathcal{U}_{\hbar}^{\wedge}$  )  be the formal quantization of $\hat{A}_{0}$ (\textit{resp.} $A$) given by $x\ast y -y\ast x=\hbar^{2}[x,y]$ for all $x,y \in \ggg_{\bar{0}}$. Equip all the above algebras with the Kazadan $\mathbb{K}^{*}$ actions arise from the good $\mathbb{Z}$-grading on $\ggg$ and $t\cdot \hbar=t\hbar$ for all $t \in \mathbb{K}^{*}$.

 Denote by $\omega$ the super even symplectic  form on $[f,\ggg]$ given by $\omega(x,y)=\chi([x,y])$. Let $V=V_{\bar{0}} \oplus V_{\bar{1}}$ be the superspace $ [f,\ggg]$ if $ \dim(\ggg(-1))$ is even. If $\dim(\ggg(-1))$ is odd, let
$V \subset [f,\ggg]$ be a superpaces with a standard basis $v_i$ with  $\omega(v_i,v_j)=\delta_{i,-j}$ for all  $i,j \in \{\pm1,\ldots, \pm (\dim([f,\ggg])-1)/2 \}$. We chose such a $V$ in the present paper  for considering the definition of W superalgebra given in \cite{WZ}.   All the statements in the present paper still valid even if we just take $V=[f,\ggg]$.

For a superspace $V$ with an even sympletic form, we denote by  $\mathbf{A}_{\hbar}(V)$ the corresponding Weyl superalgebra, see Example1.5 \cite{SX} for the definition. Specially, if $V$ is pure odd, we denote by $\mathrm{Cl}_{\hbar}(V)$ the Weyl algebra $\mathbf{A}_{\hbar}(V)$.

It  was obtained in \S2.3 \cite{Lo11} that  there is a $Q \times \mathbb{K}^{*}$-equivariant
$$\Phi_{0,\hbar}: \mathbf{A}_{\hbar}^{\wedge}(V_{\bar{0}}) \otimes \mathcal{W}_{0,\hbar}^{\wedge} \longrightarrow \mathcal{U}_{0,\hbar}^{\wedge} $$
isomorphism of quantum algebras. Moreover

\begin{prop} \label{prop2.1}
\begin{itemize}
\item[(1)]
We have  a $Q \times \mathbb{K}^{*}$-equivariant
$$ \tilde{\Phi}_{\hbar}:  \mathbf{A}_{\hbar}^{\wedge}(V_{\bar{0}}) \otimes \mathcal{\tilde{W}}_{\hbar}^{\wedge} \longrightarrow \mathcal{U}_{\hbar}^{\wedge}$$
and  $\mathbb{K}^{*}$-equivariant isomorphism
$$\Phi_{1,\hbar}: \mathrm{Cl}_{\hbar}(V_{\bar{1}})\otimes \mathcal{W}_{\hbar}^{\wedge} \longrightarrow \mathcal{\tilde{W}}_{\hbar}$$
of quantum algebras.
Finally this give us a $\mathbb{K}^{*}$-equivariant isomorphism
$$ \Phi_{\hbar}:  \mathbf{A}_{\hbar}^{\wedge}(V) \otimes \mathcal{W}_{\hbar}^{\wedge} \longrightarrow \mathcal{U}_{\hbar}^{\wedge}$$
of quantum algebras.
\item[(2)] There are isomorphisms

$$ (\mathcal{\tilde{W}}_{\hbar}^{\wedge})_{\mathbb{K}^*-{\mathrm{lf}}}/(\hbar-1)=\mathcal{\tilde{W}}; \quad (\mathcal{W}_{0,\hbar}^{\wedge})_{\mathbb{K}^*-{\mathrm{lf}}}/(\hbar-1)=\mathcal{W}_{0} \quad  \text{and} \quad
(\mathcal{W}_{\hbar}^{\wedge})_{\mathbb{K}^*-{\mathrm{lf}}}/(\hbar-1)=\mathcal{W}$$
of associative algebra.
Where, for a vector space $V$ with a $\mathbb{K}^{*}$-action, we denote by $(V)_{\mathbb{K}^*-{\mathrm{lf}}}$ the sum of
all finite dimensional $\mathbb{K}^{*}$-stable subspace of $V$.
\item[(3)] There is an embedding $ \mathfrak{q}:=\mathrm{Lie}(Q) \hookrightarrow  \mathcal{\tilde{W}}$ of Lie algebras such that
the adjoint action of $\mathfrak{q}$ coincides with the differential of the $Q$-action.
\end{itemize}
\end{prop}
\begin{proof}
(1)
Suppose that $V_{\bar{0}}$ has a basis $\{ v_{i} \}_{ 1 \leq |i| \leq l}$ with $\omega(v_i,v_j)=\delta_{i,-j}$.
The isomorphism  $\Phi_{0,\hbar}$ gives us a $Q$-equivariant embedding $\tilde{\Phi}_{\hbar}: V_{\bar{0}} \hookrightarrow \mathcal{U}_{\hbar}^{\wedge}$ with
$[\tilde{\Phi}_{\hbar}(v_i), \tilde{\Phi}_{\hbar}(v_j)]=\delta_{i,-j}\hbar$.  Now the isomorphism $\tilde{\Phi}_{\hbar}$ can be constructed as in the proof of Theorem 1.6 \cite{SX}. For the construction of isomorphism $\Phi_{1,\hbar}$, see  also Case 1 in the proof of Theorem 1.6 \cite{SX}.
The isomorphism  $\Phi_{\hbar}$ can be constructed from the embedding $ \Phi_{\hbar}: V \hookrightarrow \mathcal{U}_{\hbar}^{\wedge}$ given by
$\Phi_{\hbar}|_{V_{\bar{0}}}=\tilde{\Phi}_{\hbar}$ and $\Phi_{\hbar}|_{V_{\bar{1}}}=\Phi_{1,\hbar}$.

(2) The first isomorphism was proved in \cite{Lo11}. The remaining statements follow by  a similar argument as in the proof of Theorem 3.8 \cite{SX}.

(3) View $\mathcal{U}$ as Harish-Chandra $\mathcal{U}_0$ bimodule and   use \S2.5 \cite{Lo11}.
\end{proof}

\begin{remark}
In the proposition above we are not claiming that $\Phi_{\hbar}$ is $Q$-equivariant, although this is probably true.
\end{remark}

Proposition \ref{prop2.1} give us following  $Q \times \mathbb{K}^{*}$-equivarian version of   Theorem 4.1 \cite{SX} . 
\begin{theorem}\label{them1}
\begin{itemize}
\item[(1)]We have a $Q \times \mathbb{K}^{*}$-equivariant  embedding  $\mathcal{W}_{0} \hookrightarrow  \tilde{\mathcal{W}}$ of associative  algebras. The later is generated over the former by $\dim(\ggg_{\bar{1}})$ odd elements.
\item[(2)] Moreover we have isomorphism
$$ \Phi_1 : \tilde{\mathcal{W}} \longrightarrow \mathrm{Cl}(V_{\bar{1}})\otimes \mathcal{W}$$
 of algebras. Here $\mathrm{Cl}(V_{\bar{1}})$ is the Clifford algebra on the vector space $V_{\bar{1}}=[\mathfrak{g},f]_{\bar{1}}$
 with  symmetric two from $\chi([\cdot,\cdot])$.
\end{itemize}
\end{theorem}

\begin{proof}
 Here we can almost repeat the proof of Theorem 4.1 \cite{SX} . 
 \end{proof}

Since it is frequently used in later, it is helpful to recall the construction of  $\Phi_1$ in the following slightly general setting.

\begin{corollary}\label{coro2.4}
	For a two sided ideal  $ \tilde {\mathcal{I}}$ of  $\tilde{\mathcal{W}}$,    we have $\tilde {\mathcal{I}}=\mathrm{Cl}(V_{\bar{1}})\otimes \mathcal{I}$ . Where $\mathcal{I}$ is two sided 
ideal of $\mathcal{W}$  consist of  elements anit-commuting with  $\mathrm{Cl}(V_{\bar{1}})$.
\end{corollary}

\begin{proof}

By   Theorem   \ref{them1} (2)   there exist $x_1, \ldots, x_{\dim(V_{\bar{1}})} \in \tilde{\mathcal{W}}$  with 
$$x_i^2=1 \text {  and } x_ix_j=-x_jx_i \text{ for all  distinct $i , j  \in \{ 1, \ldots , \dim(V_{\bar{1}})\}$.}$$     
By  a   quantum analogue of Lemma 2.2(2) \cite{SX} we have  that 
$ \tilde{\mathcal{I}}$  is equal to $\mathrm{Cl}(\mathbb{K}\langle  x_1 \rangle ) \otimes \tilde{\mathcal{I}}_1$ as rings.  Here  we denote by $\tilde{\mathcal{I}}_1$   the space anti-commuting with $x_1$.
Now the corollary follows by induction to $\dim(V_{\bar{1}})$.
\end{proof}

\subsection{The maps $\bullet^{\dag}$ and $\bullet_{\dag}$}

We recall the construction of maps $\bullet^{\dag}$ and $\bullet_{\dag}$  between $\mathfrak{id}(\mathcal{W})$
and $\mathfrak{id}(\mathcal{U})$ in \cite{SX} at  first. For  $\mathcal{I} \in \mathfrak{id}(\mathcal{W})$,  denote by $\mathrm{R}_{\hbar}(\mathcal{I}) \subset \mathcal{W}_{\hbar}$ the Rees algebra
associated with $\mathcal{I}$ and $\mathrm{R}_{\hbar}(\mathcal{I})^{\wedge} \subset \mathcal{W}_{\hbar}^{\wedge}$ by completion of $\mathrm{R}_{\hbar}(\mathcal{I})$ at $0$. Let $\mathbf{A}(\mathcal{I})^{\wedge}_{\hbar}=\mathbf{A}_{\hbar}(V)^{\wedge}\otimes \mathrm{R}_{\hbar}(\mathcal{I})^{\wedge}$ and set $\mathcal{I}^{\dag}=(\mathcal{U}_{\hbar} \cap \Phi_{\hbar}(\mathbf{A}(\mathcal{I})^{\wedge}_{\hbar}))/(\hbar-1)$. For an ideal $\mathcal{J} \in \mathfrak{id}(\mathcal{U})$, denote by $\bar{\mathcal{J}}_{\hbar}$ the closure of $\mathbf{R}_{\hbar}(\mathcal{J})$ in $\mathcal{U}^{\wedge_{\chi}}_{\hbar}$. Define $\mathcal{J}_{\dag}$ to be the unique (by Proposition 3.4(3) \cite{SX}) ideal in $\mathcal{W}$  such that
$\mathbf{R}_{\hbar}(\mathcal{J}_{\dag})=\Phi_{\hbar}^{-1}(\bar{\mathcal{J}}_{\hbar}) \cap \mathbf{R}_{\hbar}(\mathcal{W}).$

A $\mathfrak{g}_{\bar{0}}$-bimodule $M$ is said to be Harish-Chandra(HC)-bimodule  if $M$ is finitely generated and the adjoint action of $\mathfrak{g}$ on $M$ is locally finite. For any two sided ideal $\mathcal{J}\subset \mathcal{U}$ (\textit{resp}. $\mathcal{I}\subset \tilde{\mathcal{W}}$), we denote by $\mathcal{J}_{\tilde{\dag}}$ (\textit{resp}. $\mathcal{I}^{\tilde{\dag}}$) image of $\mathcal{J}$
under the functor $\bullet_{\dag}$ (\textit{resp}. $\bullet^{\dag}$)  in \S3 \cite{Lo11}.  Here we view $\mathcal{J}$ and $\mathcal{I}$   as a  HC-bimodules over $\mathfrak{g}_{\bar{0}}$ and $\mathcal{W}_{0}$ respectively.
The  following lemma follows directly from  the above construction and Theorem \ref{them1}.
\begin{lemma} \label{lem2.4}
We have that $(\mathrm{Cl(V_{\bar{1}})} \otimes \mathcal{I})^{\tilde{\dag}}=\mathcal{I}^{\dag}$ and $\mathcal{I}_{\tilde{\dag}}=\mathrm{Cl}(V_{\bar{1}})\otimes \mathcal{I}_{\dag}$.
\end{lemma}

\subsection{ Properties of  $\bullet^{\dag}$ and $\bullet_{\dag}$ after \cite{SX}}

For an associative algebra $A$, we denote by $\mathrm{GK}\dim(A)$ the Gelfand-Kirillov dimension of $A$ (for the definition, see\cite{KL}).
The \textit{associated variety} $\mathbf{V}(\mathcal{J} )$ of a two sided ideal $\mathcal{J} \in \mathfrak{id}(\mathcal{U})$, is defined to be the associated variety $\mathbf{V}(\mathcal{J}_{0})$ of $ \mathcal{J}_{0}=\mathcal{J} \cap \mathcal{U}$. We say that $\mathcal{J}$ is\textit{ supported} on $\mathbf{V}(\mathcal{J})$ in this case.

\begin{lemma}\label{legkdim}
For any two sided ideal of $\mathcal{I}\subset \mathcal{U}$, we have
$$\mathrm{GK}\dim(\mathcal{U}/\mathcal{I})=\mathrm{GK}\dim(\mathcal{U}_0/\mathcal{I}_{0} )=\dim (\mathbf{V}(\mathcal{I})).$$
\end{lemma}
\begin{proof}
Note that we have the natural embedding $ \mathcal{U}_0/\mathcal{I}_{0}  \hookrightarrow \mathcal{U}/\mathcal{I}$. The first equality follows from the definition of Gelfand-Kirillov dimension (see pp.14 Definition \cite{KL} and the remark following it ) and the PBW base theorems for $\mathcal{U}(\ggg_{\bar{0}})$ and $\mathcal{U}$.
The second equality follows form Corollary 5.4 \cite{BK}.
\end{proof}

The following proposition and it's proof are super version of Theorem 1.2.2 (\rmnum{7})\cite{Los10} in a special case.
\begin{prop}\label{finitefiberprop}
For any $ \mathcal{J} \in \mathrm{Prim}_{\mathbb{O}}(\mathcal{U})$,  the set $\{  \mathcal{I}  \in \mathfrak{id} (\mathcal{W}) \mid  \text{ $\mathcal{I} $ is prime, $ \mathcal{I}^{\dag}=\mathcal{J}$  }    \}$ is exactly the minimal prime ideals containing $\mathcal{J}_{\dag}$.
\end{prop}

\begin{proof}
Suppose that $\mathcal{I}$ is prime ideal of $\mathcal{W}$ with $\mathcal{I}^{\dag}=\mathcal{J}$.  Proposition 4.5 \cite{SX} implies that  $\mathcal{J}_{\dag} \subset \mathcal{I} $.  So $\mathcal{I}$ has finite codimension in  $\mathcal{W}$.  Hence we deduce that $\mathcal{I}$ is minimal by Corollary 3.6 \cite{BK}.
Now suppose that the minimal prime ideal $\mathcal{I} \subset \mathcal{W}$ with $ \mathcal{J}_{\dag}\subset \mathcal{I}$. It is follows from Proposition 4.6 \cite{SX} that $\mathcal{J}_\dag$ has finite codimension in  $\mathcal{W}$. Thus we can see that    $ \tilde{\mathcal{I}}=\mathrm{Cl}(V_{\bar{1}}) \otimes \mathcal{I}$ has finite codimension in $\tilde{\mathcal{W}}$. Whence $\tilde{\mathcal{I}}_{0}=\mathcal{W}_{0} \cap \tilde{\mathcal{I}}$ has  finite  codimension in $\mathcal{W}_{0}$.  Since $\mathcal{I}^{\dag} \cap \mathcal{U}_{0}=(\tilde{\mathcal{I}}_{0})^{\tilde{\dag}}$, we obtain  that $\mathcal{I}^{\dag}$ is supported on $G_{\bar{0}}\cdot \chi$ by the proof of Theorem 1.2.2 (\rmnum{7})\cite{Los10}. Thus by Lemma \ref{legkdim} and Corollary 3.6 \cite{BK}, we have $\mathcal{I}^{\dag}=\mathcal{J}$.
\end{proof}

\subsection{Proof of the main result}

Now we are ready to prove the main result.

\textit{Proof of Theorem \ref{mainthm}} 

We prove the theorem  by a similar argument as in the proof of  Conjecture 1.2.1 \cite{Lo11}. Indeed, let $\mathcal{I}_{1}, \ldots, \mathcal{I}_{l}$ be the minimal prime ideal containing $\mathcal{J}_{\dag}$, for a
fixed $\mathcal{J} \in \mathrm{Prim}_{\mathbb{O}}(\mathcal{U})$. Since $ \mathrm{Cl}(V_{\bar{1}})\otimes\mathcal{I}_{1}$ is stable under $Q^{\circ}$,  $ \bigcap_{\gamma \in C_e} \gamma (\mathrm{Cl}(V_{\bar{1}}) \otimes \mathcal{I}_{1})$ is $Q$-stable.  Set $\mathcal{J}^{1}= (\bigcap_{\gamma \in C_e} \gamma (\mathrm{Cl}(V_{\bar{1}}) \otimes \mathcal{I}_{1}))^{\tilde{\dag}}$, then by Theorem 4.1.1 \cite{Lo11} we have   $(\mathcal{J}^{1})_{\tilde{\dag}}= \bigcap_{\gamma \in C_e} \gamma (\mathrm{Cl}(V_{\bar{1}}) \otimes \mathcal{I}_{1})$.
Thus we have $\mathcal{J} =(\mathcal{I}_1)^{\dag} \supset \mathcal{J}^{1}  \supset  \mathcal{J} $  (The first equality follows from Lemma \ref{legkdim} and Corollary 3.6 \cite{BK}). 
Hence $\mathcal{J}_{\tilde{\dag}}=\bigcap_{\gamma \in C_e} \gamma (\mathrm{Cl}(V_{\bar{1}}) \otimes \mathcal{I}_{1})$.  We obtain  that $\gamma (\mathrm{Cl}(V_{\bar{1}}) \otimes \mathcal{I}_{1})=\mathrm{Cl}(V_{\bar{1}})\otimes \mathcal{I}_{\gamma(1)}$ for some $\gamma(1) \in \{1,\ldots, l \}$ by Proposition 3.1.10 \cite{Di} and Corollary \ref{coro2.4}.  Thus we have
$\mathcal{I}=\bigcap_{\gamma \in C_e}\mathcal{I}_{\gamma(1)}$ by Proposition 3.1.10 \cite{Di} and Lemma \ref{lem2.4}.
Now the proof is completed by Proposition\ref{finitefiberprop}.
                                                                                             $\hfill\Box$


\subsection{ In the special  case:  $C_e=1$} \label{2.4}
	
As we have recalled  in the introduction,   the set $\mathrm{Irr}^{\mathrm{fin}}(\mathcal{W})$ is bijective with $\mathrm{Prim}^{\mathrm{cofin}}(\mathcal{W})$.  Thus  Theorem \ref{mainthm} give us  a bijection between $\mathrm{Prim}_{\mathbb{O}}(\mathcal{U})$  and $\mathrm{Irr}^{\mathrm{fin}}(\mathcal{W})$  provided $C_e=1$ .

In the case of  $\ggg=\ggg_{\bar{0}}+\ggg_{\bar{1}}$  is a basic Lie superalgebra of type \Rmnum{1}
( namely   $A(m|n)$ or $C(n)=\mathfrak{osp}(2,2n-2)$ )  Lezter established a bijection $\nu :  \mathrm{Prim} (\mathcal{U}_{0})\longrightarrow  \mathrm{Prim} (\mathcal{U})$.
It follows from the construction of $\nu$ that  the restriction give us a bijection 
$$\nu _{\mathbb{O}}:  \mathrm{Prim}_{\mathbb{O}} (\mathcal{U}_{0})\longrightarrow  \mathrm{ Prim} _{\mathbb{O}}(\mathcal{U}).$$
So  we can give a description of  $\mathrm{Irr}^{\mathrm{fin}}(\mathcal{W})$  if $C_e$ is trivial.  For all nilpotent  element in type $A(m|n)$  or at least for the regular nilotent element in type 
$C(n)$ Lie superalgebras,  the finite group  $C_e$ is trivial.

In the case of  $\ggg=\mathfrak{osp}(1,2n)$,     a  description of  $\mathrm{Prim}(\mathcal{U})$   given in  Theorem A and B \cite{Mu}.  Any primitive ideal  is  the annihilator  of  a  simple highest weight module  $\hat{L}(\lambda)$ for some  $\lambda \in \mathfrak{h}^{*}$.  The poset structure describing   $\mathrm{Prim}(\mathcal{U})$   is  exactly same as  that $\mathrm{Prim} (\mathcal{U}_{0})$.   It is straightforward to check that  $\hat{L}(\lambda)$ is supported on $\bar{\mathbb{O}}$ if and only if so is $L(\lambda)$. Thus we show that Theorem \ref{mainthm} give us a description of $\mathrm{Irr}^{\mathrm{fin}}(\mathcal{W})$ provided $C_{e}=1$.

\section*{Acknowledgements}
The author is partially supported by NFSC (grant No.11801113) and  RIMS, an international joint usage/research center located in Kyoto University. This work is motivated by communications with Arakawa and is written during the author's visit to him at RIMS. The author indebted much to Arakawa for many helpful discussions.  The author also thank the helpful communications from Bin Shu and Yang Zeng.

\end{document}